\documentclass{amsart}
\usepackage{amsmath,amssymb,amsthm,amsfonts,amscd,mathrsfs,mathtools,enumerate}
\usepackage{tikz,tikz-cd,caption}
\usepackage[all]{xy}
\usepackage[hyphens]{url}
\usepackage{hyperref}
\usepackage{caption}
\usetikzlibrary{quotes,arrows}

\theoremstyle{plain}
    \newtheorem{thm}{Theorem}[section]
    \newtheorem{lem}[thm]   {Lemma}
    \newtheorem{cor}[thm]   {Corollary}
    \newtheorem{prop}[thm]  {Proposition}
    
\theoremstyle{definition}
    \newtheorem{defn}[thm]  {Definition}
    \newtheorem{prob}[thm] {Problem}

    \newtheorem{ex}[thm]{Example}
    \newtheorem{rem}[thm]{Remark}

\def\max{\mathrm{max}}

\def\cat{\mathsf{cat}}
\def\secat{\mathsf{secat}}

\def\dim{\mathrm{dim}}
\def\ker{\operatorname{ker}}
\def\nil{\operatorname{nil}}

\newcommand{\be}{\begin{enumerate}}
\newcommand{\ee}{\end{enumerate}}
\newcommand{\R}{\mathbb{R}}
\newcommand{\Z}{\mathbb{Z}}
\newcommand{\N}{\mathbb{N}}
\newcommand{\C}{\mathbb{C}}

\newcommand{\TC}{{\sf TC}}

\newcommand{\cld}{{\sf cd}}

\newcommand{\paths}[1]{P#1}
\newcommand{\eval}[1]{\pi_{#1}}
\newcommand{\evalG}[1]{p_{#1}}
\newcommand{\evaln}[1]{\pi_{n,#1}}

\begin{document}

\title[Equivariant TC]{Equivariant topological complexities}

\author{Mark Grant}

\address{Institute of Mathematics,
Fraser Noble Building,
University of Aberdeen,
Aberdeen AB24 3UE,
UK}

\email{mark.grant@abdn.ac.uk}

\date{\today}


\begin{abstract} Many mechanical systems have configuration spaces that admit symmetries. Mathematically, such symmetries are modelled by the action of a group on a topological space. Several variations of topological complexity have emerged that take symmetry into account in various ways, either by asking that the motion planners themselves admit compatible symmetries, or by exploiting the symmetry to motion plan
between functionally equivalent configurations. We will survey the main definitions due to Colman-Grant, Lubawski-Marzantowicz, B\l{}aszczyk-Kaluba and Dranishnikov, and some related notions. We conclude with a short list of open problems.
\end{abstract}

\maketitle
\section{Introduction}\label{sec:intro}

As seen in previous chapters, the topological complexity $\TC(X)$ provides an interesting measure of the complexity (from a topological perspective) of the motion planning problem for a mechanical system with $X$ as its configuration space. In many naturally occurring examples, the configuration space admits non-trivial symmetries, which one may wish to take into account when designing motion planning algorithms. This leads to several variations on the notion of topological complexity, which will be surveyed in this chapter.

Symmetry is a central concept in mathematics, and symmetries of topological spaces are particularly well studied. The set of symmetries forms a \emph{group} $G$, which \emph{acts} on the configuration space $X$. Precise definitions will follow in the next section; for the purposes of this introduction, we think it best to illustrate the concept with some relevant examples.

\begin{ex}\label{ex:Intro1}
Consider a planar mechanism (such as a robot arm), one component of which is anchored to a point in the plane. Denote its configuration space by $Y$. Now add an extra revolute joint, giving one more degree of freedom. We imagine creating a spatial mechanism by basing the anchor of the arm to a rotating platform or circular track in $3$-space. The configuration space of this new mechanism is the topological product $X:=S^1\times Y$, where $S^1$ is the unit circle.

\begin{minipage}{0.49\linewidth}
\begin{center}
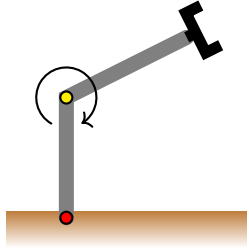

\begin{tikzpicture}
[line cap=round,scale=0.8]
\shade[top color=brown,bottom color=white] (-1,-.5) rectangle (3,0.1);
\draw[line width=2mm,gray](0,0)--(0,2)--(2,3);
\draw[thick,fill=red] (0,0) circle [radius=1mm];
\draw[thick,fill=yellow] (0,2) circle [radius=1mm];
\draw[->,thick](0,2) ++(240:5mm) arc (240:-60:5mm);
\draw[line width=1.5mm,line cap=butt,-[](2,3)--(2.4,3.2);
\path (-3,-0.7) rectangle (5,4);
\end{tikzpicture}
\captionsetup{width=0.8\linewidth}
\captionof{figure}{A configuration in $Y$}
\end{center}
\end{minipage}
\begin{minipage}{0.49\linewidth}
\begin{center}
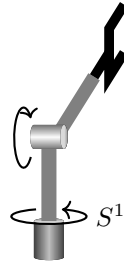

\begin{tikzpicture}
[yscale=0.8]
\fill[left color=gray!50!black,right color=gray!50!black,middle color=gray!50,shading=axis,opacity=0.25] (.2,0) -- (.2,.6) arc (360:180:.2cm and 0.05cm) -- (-.2,0) arc (180:360:.2cm and 0.05cm);
\fill[top color=gray!90!,bottom color=gray!2,middle color=gray!30,shading=axis,opacity=0.25] (0,.6) circle (.2cm and 0.05cm);
\draw[line width=2mm,gray](0,.6)--(0,1.8);
\draw[line width=1.8mm,gray](.1,2)--(.6,3);
\draw[line width=1.3mm](.6,3)--(.8,3.4)--(.8,3.0)--(1,3.4);
\draw[line width=1.3mm](.8,3.4)--(.8,3.8)--(1,4.2);
\fill[white](1,3) rectangle (2,4.5);
\fill[top color=gray!90!black,bottom color=gray!90!black,middle color=gray!2,shading=axis,opacity=1] (-.2,2.2) -- (.2,2.2) arc (450:270:.05cm and 0.2cm) -- (-.2,1.8) arc (270:90:.05cm and 0.2cm);
\fill[right color=gray!90!,left color=gray!30,middle color=gray!30,shading=axis,opacity=1] (.2,2) circle (.05cm and 0.2cm);
\draw[->,thick](-.1,.6)  ++(110:.2075cm) arc (110:430:.5cm and .125cm);
\draw[->,thick](-.4,1.8)  ++(320:.2075cm) arc (320:40:.125cm and .5cm);
\node[right] at (.5,.7){$S^1$};
\end{tikzpicture}
\captionsetup{width=0.9\linewidth}
\captionof{figure}{A configuration in $S^1\times Y$}
\end{center}
\end{minipage}

\vspace{1em}

This new configuration space has an obvious rotational symmetry. Mathematically speaking, the \emph{group} $S^1$ (of unit complex numbers, say) \emph{acts} on $X$ by a formula $z(w,y)=(zw,y)$ for $z,w\in S^1$ and $y\in Y$. It seems reasonable to suggest that motion planners in $X$ should be \emph{invariant} under this symmetry, namely that the motion from $(zw_1,y_1)$ to $(zw_2,y_2)$ should be the $z$-rotation of the motion from $(w_1,y_1)$ to $(w_2,y_2)$, for all $z\in S^1$ and $(w_1,y_1),(w_2,y_2)\in S^1\times Y$. Whether or not this condition is enforced, it could play a role in the design of motion planning algorithms.
\end{ex}

\begin{ex}[from \cite{LubMar}]\label{ex:Intro2}
Consider a planar robot arm with $n$ revolute joints and a ``hand" at the end of the arm for grasping objects. Ignoring collisions, the configuration space $X=(S^1)^n$ is an $n$-dimensional torus. Pictured below is the case $n=6$. There is an obvious symmetry of the arm which exchanges the two ``fingers" of the ``hand".

\begin{minipage}{0.49\linewidth}
\begin{center}
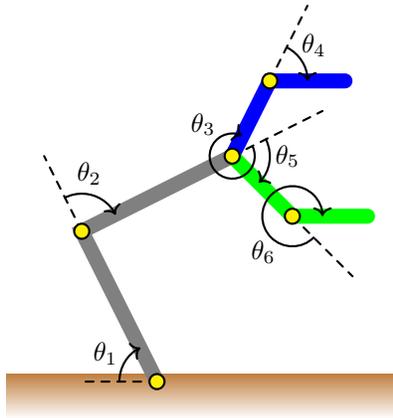

\begin{tikzpicture}
[line cap=round,scale=1]
\shade[top color=brown,bottom color=white] (-2,-.5) rectangle (3.2,0.1);
\draw[thick, dashed](0,0)--(-1,0);
\draw[thick, dashed](-1,2)--(-1.5,3);
\draw[thick, dashed](1,3)--(2.2,3.6);
\draw[thick, dashed](1.5,4)--(2,5);
\draw[thick, dashed](1.8,2.2)--(2.6,1.4);
\draw[line width=2mm,gray](0,0)--(-1,2)--(1,3);
\draw[line width=2mm,blue](1,3)--(1.5,4)--(2.5,4);
\draw[line width=2mm,green](1,3)--(1.8,2.2)--(2.8,2.2);
\draw[thick,fill=yellow] (0,0) circle [radius=1mm];
\node[above left] at (-.4,0.1){$\theta_1$};
\node[above] at (-.9,2.5){$\theta_2$};
\node[above left] at (.9,3.15){$\theta_3$};
\node[above right] at (1.8,4.2){$\theta_4$};
\node[right] at (1.45,3){$\theta_5$};
\node[below left] at (1.7,2){$\theta_6$};
\draw[thick,fill=yellow] (-1,2) circle [radius=1mm];
\draw[thick,fill=yellow] (1,3) circle [radius=1mm];
\draw[thick,fill=yellow] (1.5,4) circle [radius=1mm];
\draw[thick,fill=yellow] (1.8,2.2) circle [radius=1mm];
\draw[->,thick](0,0) ++(180:5mm) arc (180:116:5mm);
\draw[->,thick](-1,2) ++(117:5mm) arc (117:26:5mm);
\draw[->,thick](1,3) ++(26:5mm) arc (26:-45:5mm);
\draw[->,thick](1,3) ++(26:3mm) arc (26:-296:3mm);
\draw[->,thick](1.8,2.2) ++(315:4mm) arc (-45:-360:4mm);
\draw[->,thick](1.5,4) ++(63:5mm) arc (63:0:5mm);
\end{tikzpicture}
\captionsetup{width=0.7\linewidth}
\captionof{figure}{A configuration $x=(\theta_1,\ldots, \theta_6)$ in $X$}
\end{center}
\end{minipage}
\begin{minipage}{0.49\linewidth}
\begin{center}
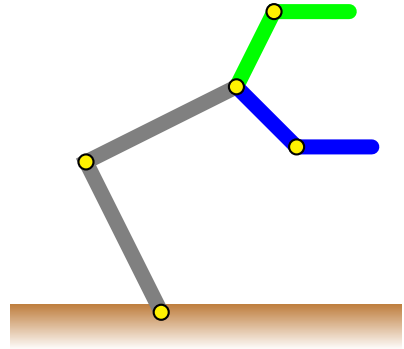

\begin{tikzpicture}
[line cap=round,scale=1]
\shade[top color=brown,bottom color=white] (-2,-.5) rectangle (3.2,0.1);
\draw[line width=2mm,gray](0,0)--(-1,2)--(1,3);
\draw[line width=2mm,green](1,3)--(1.5,4)--(2.5,4);
\draw[line width=2mm,blue](1,3)--(1.8,2.2)--(2.8,2.2);
\draw[thick,fill=yellow] (0,0) circle [radius=1mm];
\draw[thick,fill=yellow] (-1,2) circle [radius=1mm];
\draw[thick,fill=yellow] (1,3) circle [radius=1mm];
\draw[thick,fill=yellow] (1.5,4) circle [radius=1mm];
\draw[thick,fill=yellow] (1.8,2.2) circle [radius=1mm];
\path (-2,-.5) rectangle (3.2,5);
\end{tikzpicture}
\captionsetup{width=0.7\linewidth}
\captionof{figure}{The symmetric configuration $\tau x$}
\end{center}
\end{minipage}

\vspace{1em}
 In terms of the joint angles, this is given by
\[
(\theta_1,\theta_2,\theta_3,\theta_4,\theta_5,\theta_6)\mapsto (\theta_1,\theta_2,\theta_5,\theta_6,\theta_3,\theta_4) 
\]
Mathematically speaking, the cyclic group of order two $C_2=\{1,\tau\}$ acts on $X$ by permuting coordinates. In this example, a configuration $x$ and its image $\tau x$ under the involution are physically different, but \emph{functionally equivalent}---the spatial configuration of the robot is the same up to a relabeling of its parts. This suggests that rather than having to navigate a path between $x$ and $\tau x$ to perform tasks, the robot should instead be able to ``recalibrate" and exploit the symmetries to make its task easier. 
\end{ex}  

These two examples are of a very different nature. Example \ref{ex:Intro1} is supposed to illustrate that it may be natural to ask for motion planners that are themselves symmetric. Example \ref{ex:Intro2} illustrates that there may be functionally equivalent symmetric configurations, in which case there is no need to navigate a path in the configuration space between them in order to perform tasks. There are many possible design requirements, and the mathematical discipline of equivariant topology provides an unending supply of examples of configuration spaces with symmetry to illustrate the various phenomena that can occur. This aspect, combined with the differing motivations of the researchers involved, explains why several different definitions of topological complexity with symmetry have emerged in a relatively short space of time. The catalyst for this seems to have been the Arbeitsgemeinschaft in Topological Robotics held at the MFO in 2010, organized by Michael Farber, Jes\'us Gonz\'alez and Dirk Sch\"utz, at which the author and Hellen Colman began discussing the ideas that led to \cite{ColmanGrant}. The subject also featured heavily at the Mini-Workshop ``Topological Complexity and Related Topics" held at the MFO in 2016, organized by the author, Greg Lupton and Lucile Vandembroucq. We would like to thank the MFO and the various organizers and co-organizers for providing such a stimulating atmosphere.

There already exists an excellent survey article with the same title as this one, written by \'Andres \'Angel and Hellen Colman \cite{AngelColmanSurvey}. While there is necessarily considerable overlap in content between that article and this chapter, we have tried to add to the narrative, and to cover developments since the publication of \cite{AngelColmanSurvey}. We also include one new result, a cohomological lower bound for Dranishnikov's strongly equivariant topological complexity (Proposition \ref{prop:lowerTC^*_G}). We conclude our survey with a list of what we believe to be interesting research problems in the area.

The structure of this survey is as follows. In Section \ref{S:actions} we collect some basic definitions about group actions which we will need, so as not to interrupt the flow later on. Sections \ref{S:TC_G}, \ref{S:TC^G}, \ref{S:TC^*_G} and \ref{S:TCeff} focus on, respectively: the equivariant topological complexity $\TC_G(X)$ of Colman and the author \cite{ColmanGrant}; the invariant topological complexity $\TC^G(X)$ of Lubawski and Marzantowicz \cite{LubMar}; the strongly equivariant topological complexity $\TC^*_G(X)$ of Dranishnikov; and the effective topological complexity $\TC^{G,\infty}(X)$ of B\l{}aszczyk and Kaluba \cite{BlaKal}. The final Section \ref{S:Problems} contains our list of open problems.

We adopt the convention of this volume, that all $\TC$ and $\cat$ invariants are reduced. Thus some of the formulae from the original references \cite{ColmanGrant,LubMar,BlaKal} have been adjusted to account for this.

\section{Group actions}\label{S:actions}

In this section we briefly recall some of the basic definitions from the theory of group actions. The reader unfamiliar with the theory might wish to consult one of the classic texts on the subject for further details, for example \cite{Bredon} or \cite{tomDieck}. Further definitions will be supplied in the text as required. 

Recall that a \emph{group} is a set $G$ equipped with an associative binary operation $\mu: G\times G\to G$ that admits an identity element $e\in G$ satisfying $\mu(e,g)=g=\mu(g,e)$ for all $g\in G$, and such that each $g\in G$ admits an inverse $g^{-1}\in G$ satisfying $\mu(g,g^{-1})=e=\mu(g^{-1},g)$. The inverse is easily shown to be unique, hence we may define a function $\nu:G\to G$, called \emph{inversion}, by setting $\nu(g)=g^{-1}$. A \emph{topological group} is a group $G$ equipped with a topology such that the operation $\mu: G\times G\to G$ and inversion $\nu: G\to G$ are continuous functions. A \emph{Lie group} is a group $G$ equipped with the structure of a smooth manifold such that $\mu:G\times G\to G$ and $\nu: G\to G$ are smooth maps.

\begin{defn} \label{def:action}
Let $X$ be a space (respectively, smooth manifold), and let $G$ be a toplogical (respectively,  Lie) group. An \emph{action} of $G$ on $X$ is a continuous (respectively, smooth) function $\alpha: G\times X\to X$ satisfying:
\begin{enumerate}
\item $\alpha(g,\alpha(h,x))=\alpha(\mu(g,h),x)$ for all $g,h\in G$ and $x\in X$;
\item $\alpha(e,x)=x$ for all $x\in X$, where $e\in G$ is the identity element.
\end{enumerate}
\end{defn}

In theoretical work one often drops the names of the operations $\mu$ and $\alpha$ from the notation, and writes these using concatenation and appropriate parentheses. Thus the associativity property from the definition of a group becomes $g(hk)=(gh)k$, and property (1) from the definition of an action becomes $g(hx)=(gh)x$.

\begin{defn} Let $X$ and $Y$ be spaces, each with an action of $G$. A continuous function $\phi :X\to Y$ is called \emph{$G$-equivariant} if $\phi(gx)=g\phi(x)$ for all $g\in G$ and $x\in X$.
\end{defn}

For brevity, we will refer to a space equipped with an action of a group $G$ as a \emph{$G$-space}, and to a continuous $G$-equivariant function as a \emph{$G$-map}.

\begin{defn} Let $X$ be a $G$-space, and let $A\subseteq X$ be a subset. The \emph{$G$-saturation} of $A$ is the set
\[
GA:=\{ga \mid g\in G, a\in A\}\subseteq X.
\]
If $GA=A$, then $A$ is called \emph{$G$-invariant}.

As a special case of the above, let $A=\{x\}$ be a singleton. We write $Gx$ instead of $G\{x\}$, and call this the \emph{$G$-orbit} of $x\in X$.
\end{defn}

If $A\subseteq X$ is $G$-invariant then the $G$-action on $X$ restricts to $A$, making $A$ a $G$-space. The inclusion $\iota_A:A\hookrightarrow X$ is then a $G$-map, and $(X,A)$ is called a \emph{pair of $G$-spaces}.

A \emph{subgroup} of $G$ is a subset $H\subseteq G$ such that the binary operation $\mu: G\times G\to G$ restricts to an operation $\mu|_{H\times H}: H\times H \to H$ satisfying the group axioms (it suffices to check that $\mu(g,h^{-1})\in H$ whenever $g,h\in H$). A subgroup of a topological group is a topological group, and a \emph{closed} subgroup of a Lie group is a Lie group. We will write $H\leq G$ to signify that $H$ is a subgroup of $G$. Given $H\leq G$ and a $G$-action on a space $X$, restriction gives an $H$-action on $X$ of the same type.

A subgroup $H\leq G$ is \emph{normal} if $ghg^{-1}\in H$ whenever $g\in G$ and $h\in H$. Under this condition, we can form the \emph{quotient group} $G/H$ whose elements are the cosets $gH$ and whose binary operation is given by $\mu(gH,g'H)=gg'H$.

\begin{defn} Let $X$ be a $G$-space, and let $H\leq G$ be a subgroup. The \emph{$H$-fixed point set} of $X$ is
\[
X^H:=\{ x\in X \mid hx=x\mbox{ for all }h\in H\}\subseteq X.
\]
\end{defn}

Note that a $G$-map $f:X\to Y$ induces a map
\[
f^H:=f|_{X^H}:X^H\to Y^H
\]
for every subgroup $H\leq G$.

\begin{defn}
Let $X$ be a $G$-space, and let $x\in X$. The \emph{isotropy subgroup} of $x$ is
\[
G_x:=\{g\in G \mid gx=x\} \leq G.
\]
\end{defn}

Note that if $f:X\to Y$ is a $G$-map then $G_x \leq G_{f(x)}$ for all $x\in X$.

Many of the results in this survey require some additional hypotheses about the group $G$, the space $X$, or the action of $G$ on $X$. Common hypotheses including assuming that $G$ is a compact Lie group, or that $G$ is a finite group (with the discrete topology; equivalently, a compact Lie group of dimension $0$). Often arguments require some point-set assumptions on $X$, such as metrizability, being an ANR (absolute neighbourhood retract), or being (completely) normal, but it is not clear whether the result is true without such hypotheses. When results as they originally appeared in the literature admit straightforward generalizations, we have tried to indicate this. We do not claim to have found the most general form of each result, however.

Common assumptions on the action of $G$ on $X$ include:

\begin{itemize}
\item freeness: an action is  \emph{free} if $G_x=\{e\}$ for all $x\in X$;
\item the existence of a fixed point $x_0\in X^G$;
\item properness: an action is \emph{proper} if for any compact subset $C\subseteq X$ the set
\[
\{g\in G \mid C\cap gC\neq\emptyset\}
\]
is compact.
\end{itemize}
It is known that any smooth action of a compact Lie group $G$ on a manifold $X$ is proper, and is obvious that any action of a finite group $G$ is proper. Note that properness implies that every isotropy group $G_x$ is compact.

\begin{defn}
Let $X$ be a $G$-space. The \emph{orbit space} of a $G$-space $X$ is the set of orbits
\[
X/G :=\{ Gx \mid x\in X\},
\]
given the quotient topology via the surjection $X\to X/G$ sending $x$ to its orbit $Gx$.
\end{defn}

For general actions the orbit space can be fairly upsetting as a topological space. For instance, if $G$ is non-compact, then $X/G$ can fail to be Hausdorff, even if $X$ is Hausdorff. However, if $G$ is a compact Lie group acting freely on a smooth manifold $X$, then $X/G$ is in a natural way a smooth manifold and the orbit map $X\to X/G$ is a smooth submersion.

We now turn to the basic definitions of equivariant homotopy theory.
 
\begin{defn} Let $X$ and $Y$ be $G$-spaces, and let $\phi,\psi:X\to Y$ be $G$-maps. A \emph{$G$-homotopy} from $\phi$ to $\psi$ is a $G$-map $H: X\times [0,1]\to Y$ such that $H(x,0)=\phi(x)$ and $H(x,1)=\psi(x)$ for all $x\in X$. Here $G$ acts on $X\times[0,1]$ via $g(x,t)=(gx,t)$. If such a $G$-homotopy exists, then we say $\phi$ and $\psi$ are \emph{$G$-homotopic} and write $\phi\simeq_G \psi$. This defines an equivalence relation on the set of $G$-maps from $X$ to $Y$. 
\end{defn}

In other words, two $G$-maps are $G$-homotopic if one can be continuously deformed into the other through $G$-maps.

\begin{defn}
We say that two $G$-spaces $X$ and $Y$ are {\em $G$-homotopy equivalent}, and write $X\simeq_G Y$, if there exist $G$-maps $\phi:X\to Y$ and $\psi:Y\to X$ such that $\psi\circ\phi\simeq_G\operatorname{Id}_X$ and $\phi\circ\psi\simeq_G \operatorname{Id}_Y$.
\end{defn}

The reader familiar with ordinary homotopy theory will easily be able to provide the definitions of \emph{$G$-fibration} and \emph{$G$-cofibration}.

Our working definition of a \emph{principal $G$-bundle} is as follows. Suppose $E$ is a free $G$-space such that the orbit projection map $p:E\to E/G=:B$ is a fiber bundle with fiber $G$. Then $p:E\to B$ is a principal $G$-bundle. Not every free action gives rise to a principal $G$-bundle, but free smooth actions of compact Lie groups on smooth manifolds do. There always exists a universal principal $G$-bundle $EG\to BG$, unique up to bundle isomorphism, in which the total space $EG$ is a free contractible $G$-space.

Given a $G$-space $X$ and a principal $G$-bundle $E\to B$, we can form the \emph{associated bundle with fiber $X$}, whose total space $E\times_G X$ is the orbit space of $E\times X$ under the \emph{diagonal action} of $G$ given by $g(e,x)=(ge,gx)$, and whose bundle projection $E\times_G X\to B$ is given by $[e,x]\mapsto [e]$. When $E=EG$, the total space $EG\times_G X$ is called the \emph{homotopy orbit space of $X$} (or the \emph{Borel construction on $X$}).

\begin{rem}
Definition \ref{def:action} is of a \emph{left action} of $G$ on $X$. There is a corresponding definition of \emph{right action}. We'll do everything here with left actions. It's usually an easy matter to translate between the two.
\end{rem}

\section{Equivariant topological complexity}\label{S:TC_G}

The first version of topological complexity to take symmetries into account was introduced by Colman and the author in \cite{ColmanGrant}, and called the \emph{equivariant topological complexity}. We let $\paths{X}$ denote the space of continuous paths in $X$, equipped with the compact-open topology. Note that if $G$ acts on $X$, then it also acts on $\paths{X}$ via sending a path $\gamma$ to the path $g\gamma$ defined by $(g\gamma)(t)=g\gamma(t)$, and on $X\times X$ via the \emph{diagonal action} $g(x,y)=(gx,gy)$. The end-point fibration $\eval{X}:\paths{X}\to X\times X$ is then a $G$-map.

\begin{defn}[{\cite[Definition 5.1]{ColmanGrant}}]\label{def:TC_G}
Let $X$ be a $G$-space. The \emph{equivariant topological complexity} of $X$, denoted $\TC_G(X)$, is defined to be the minimal integer $k$ such that $X\times X$ can be covered by $G$-invariant open sets $U_0,\ldots , U_k$, each admitting a $G$-map $s_i:U_i\to \paths{X}$ such that $\eval{X}\circ s_i=\iota_{U_i}:U_i\hookrightarrow X\times X$.
\end{defn}

\begin{rem} The equivariant topological complexity $\TC_G(X)$ equals $\secat_G(\eval{X})$, the \emph{equivariant sectional category} of the $G$-fibration $\eval{X}$ \cite[Definition 4.1]{ColmanGrant}. Other instances of this more general notion are considered in the subsequent articles \cite{Grant} and \cite{GrantMeirPatchkoria}. 
\end{rem}

Note that removing all occurences of the group $G$ from Definition \ref{def:TC_G} recovers the definition of $\TC(X)$. Two easy observations follow:
\begin{enumerate}
\item If the action is \emph{trivial}, meaning that $gx=x$ for all $g\in G$ and $x\in X$, then $\TC_G(X)=\TC(X)$;
\item $\TC_G(X)\geq \TC(X)$ for any $G$-space $X$. 
\end{enumerate}
Property (1) is something we would hope for and expect of any variant of $\TC(X)$ which takes symmetry into account. Property (2) tells us that forcing our motion planning rules to be symmetric does not reduce the complexity of the motion planning task.

\begin{prop}[{\cite[Theorem 5.2]{ColmanGrant}}]
The equivariant topological complexity is $G$-homotopy invariant. That is, if $X\simeq_G Y$, then $\TC_G(X)=\TC_G(Y)$.
\end{prop}

Thus the computation of equivariant topological complexity may be amenable to the tools of equivariant algebraic topology. For example, there is a lower bound for $\TC_G(X)$ coming from equivariant cohomology. 

\begin{thm}[{\cite[Theorem 5.15]{ColmanGrant}}]\label{thm:lowerTC_G}
Let $h^*_G(-)$ denote a multiplicative $G$-equivariant cohomology theory, and let $\Delta:X\to X\times X$ be the diagonal map. Then 
\[
\TC_G(X)\geq\nil\ker \big(\Delta^*:h^*_G(X\times X)\to h^*_G(X)\big).
\]
\end{thm}

The exact definition of ``multiplicative $G$-equivariant cohomology theory'' is deliberately left ambiguous, but examining the proof one sees that the only properties required of $h^*_G(-)$ are $G$-homotopy invariance, the long exact sequence of a pair of $G$-spaces and naturality of (relative) cup products. For example we may take as $h^*_G(-)$ the Borel cohomology $H^*_G(-;R):=H^*(EG\times_G -;R)$, that is the ordinary cellular or singular cohomology of the homotopy orbit space with coefficients in an arbitrary commutative ring $R$.

Also by analogy with the non-equivariant case, we have the following product formula, which was first stated as \cite[Theorem 4.2]{GGTX} under the assumption that $X$ and $Y$ are smooth $G$-manifolds.

\begin{thm}\label{thm:productTC_G}
Let $G$ be a compact Lie group, and let $X$ and $Y$ be paracompact $G$-spaces. Then
\[
\TC_G(X\times Y)\leq \TC_G(X) + \TC_G(Y).
\]
where $X\times Y$ is given the diagonal $G$-action.
\end{thm}

\begin{proof} Examining the proof of the product inequality $\TC(X\times Y)\leq\TC(X)+\TC(Y)$ for ordinary topological complexity given in \cite[Theorem 11]{Far03}, one sees that all that is required to make it equivariant is the existence of $G$-equivariant partitions of unity subordinate to $G$-invariant open covers. When $G$ is compact Lie, this can be arranged by averaging over non-equivariant partitions of unity, as in \cite[Lemma 3.2]{Grant}.
\end{proof}

The following results relate the equivariant topological complexity of a $G$-space to the (equivariant) topological complexities of the fixed point sets. The original statements include the assumption that $H$ and $K$ be closed subgroups, which appears on closer inspection to be unnecessary.

\begin{prop}[{\cite[Proposition 5.3]{ColmanGrant}}]
Let $H$ and $K$ be subgroups of $G$ such that $X^H\subseteq X$ is $K$-invariant. Then $\TC_K(X^H)\leq \TC_G(X)$.
\end{prop}

\begin{cor}[{\cite[Corollary 5.4]{ColmanGrant}}]\label{cor:fixed}
For any subgroups $H$ and $K$ of $G$ one has $\TC(X^H)\leq \TC_G(X)$ and $\TC_K(X)\leq \TC_G(X)$.
\end{cor}

Recall that $\TC(Y)$ fails to be finite if $Y$ is not path-connected. In particular, if for some subgroup $H$ of $G$ the the fixed point set $X^H$ is not path-connected, then by Corollary \ref{cor:fixed},
\[
\infty=\TC(X^H)\leq \TC_G(X).
\]
Thus the equivariant topological complexity may be infinite, and the topological complexity finite. 

\begin{ex}\label{ex:S1reflection}
Let $G=C_2$ be cyclic of order two, acting on $S^1\subseteq \C$ by complex conjugation. Then $\TC_G(S^1)=\infty$, while $\TC(S^1)=1$.
\end{ex}

\begin{defn}
A $G$-space $X$ is called \emph{$G$-connected} if the fixed point set $X^H$ is path-connected for every subgroup $H\leq G$. 
\end{defn}

Thus for $\TC_G(X)$ to be finite, it is necessary that $X$ be $G$-connected. The notion of $G$-connectedness is also involved in inequalities relating the equivariant topological complexity to the equivariant Lusternik--Schnirelmann category. We recall the definition of the latter.

\begin{defn}[{\cite{Fadell,Marzantowicz,ClappPuppe91}}]
Let $X$ be a $G$-space. The \emph{equivariant (Lusternik--Schnirelmann) category} of $X$, denoted $\cat_G(X)$, is defined to be the minimal integer $k$ such that $X$ can be covered by $G$-invariant open sets $U_0,\ldots , U_k$ such that each inclusion $\iota_{U_i}:U_i\hookrightarrow X$ is $G$-homotopic to a map with values in a single orbit.
\end{defn}

\begin{prop}[{\cite[Propositions 5.6, 5.7]{ColmanGrant}}]\label{prop:catGbounds}
Let $X$ be a $G$-space.
\begin{enumerate}[(a)]
\item If $X$ is $G$-connected, then $\TC_G(X)\leq \cat_G(X\times X)$.
\item If $H\subseteq G$ is the isotropy group of some $x\in X$, then $\cat_H(X)\leq \TC_G(X)$.
\end{enumerate}
\end{prop}

\begin{cor}[{\cite[Corollary 5.8]{ColmanGrant}}]\label{cor:catGbounds}
Let $X$ be a completely normal $G$-connected $G$-space with a fixed point $x\in X^G$. Then
\[
\cat_G(X)\leq \TC_G(X)\leq 2\cat_G(X).
\]
\end{cor}

\begin{rem}
To deduce Corollary \ref{cor:catGbounds} from Proposition \ref{prop:catGbounds} one needs a product inequality for the equivariant category for a diagonal action. We note here that the statements of \cite[Proposition 3.2]{Cicortas} and \cite[Theorem 3.15]{ColmanGrant} are missing assumptions on the fixed point sets. The correct statement appears as Theorem 2.23 in the paper \cite{BayehSarkar15} by Bayeh and Sarkar, which also includes counter-examples to \cite[Theorems 3.15, 3.16]{ColmanGrant}.
\end{rem}
 
One of the motivations for introducing the equivariant topological complexity was to give upper bounds for the topological complexity of total spaces of fiber bundles. Recall that given a $G$-space $X$ and a principal $G$-bundle $E\to B$, one gets an associated bundle $p:E\times_G X\to B$ whose fiber is $X$ and whose total space is the orbit space of the diagonal $G$-action on $E\times X$. 

\begin{thm}[{\cite[Theorem 5.16]{ColmanGrant}}]\label{thm:TCFib}
Let $X$ be a $G$-space, and let $E\to B$ be a numerable principal $G$-bundle. Then
\[
\TC(E\times_G X) <(\TC(B)+1)(\TC_G(X)+1).
\]
\end{thm}

This theorem was applied in the paper \cite{GGTX} to obtain upper bounds on the topological complexity of projective product spaces.

There is also a direct inequality relating the equivariant and parametrized topological complexities. The parametrized topological complexity $\TC[p:Y\to B]$ of a fibre bundle $p:Y\to B$ was defined elsewhere in this volume.

\begin{thm}[{\cite[Theorem 3.4]{FarberOprea}}]\label{thm:equivparam}
For the associated bundle $p:E\times_G X\to B$ as above, one has
\[ 
\TC[p:E\times_G X\to B] \leq \TC_G(X).
\]
\end{thm}

We mention in passing that the proof of Theorem \ref{thm:TCFib} given in \cite{ColmanGrant} can be modified (perhaps simplified) to prove the stronger inequality
\[
\TC(E\times_G X) <(\TC(B)+1)(\TC[p:E\times_G X\to B]+1).
\]
We leave the details to the interested reader.

\begin{rem}
Bayeh and Sarkar \cite{BayehSarkar20} have defined the \emph{sequential equivariant topological complexity} $\TC_{r,G}(X)$ for a $G$-space $X$ and natural number $r\ge2$. Many (if not all) of the results for $r=2$ in this section for admit natural generalizations for higher $r$. In particular, in \cite{FarberOprea} the authors treat sequential parametrized topological complexity and Theorem \ref{thm:equivparam} is stated in the sequential setting.
\end{rem}

\section{Invariant topological complexity}\label{S:TC^G}

Although the definition of $\TC_G(X)$ is fairly natural from a mathematical perspective, it has a number of drawbacks. Firstly, requiring our motion planners to be equivariant only makes the motion planning task more complex, as evidenced by the inequality $\TC(X)\leq \TC_G(X)$ which can certainly be strict (see Example \ref{ex:S1reflection}). One might hope instead to exploit the symmetries to ease the motion planning task. Secondly, from a mathematical perspective it would be desirable if for a free $G$-action the equivariant topological complexity were equal to the topological complexity of the orbit space. However this is not the case, as the following example shows.

\begin{ex}\label{ex:S1rotation}
Let $G=X=S^1$ be the circle group, acting on itself by rotation. The orbit space $X/G$ is a point, hence $\TC(X/G)=0$. However $1=\TC(X)\leq\TC_G(X)$. Therefore $\TC_G(X)\neq \TC(X/G)$ in this example, although the action is free.
\end{ex}

One reason underlying this phenomenon is that the definition of $\TC_G(X)$ involves the diagonal action of $G$ on $X\times X$, while the definition of $\TC(X/G)$ involves $X/G\times X/G$, which is the orbit space of the action of $G\times G$ on $X\times X$.  
Motivated by this, Lubawski and Marzantowicz introduced in \cite{LubMar} a different variant of topological complexity with symmetry, called the \emph{invariant topological complexity}, which takes into account the action of $G\times G$. Note that the path space $\paths{X}$ does not carry a natural action of $G\times G$, so Lubawaski and Marzantowicz introduce the space
\[
\paths{X}\times_{X/G}\paths{X} = \{(\gamma,\delta)\in\paths{X}\times\paths{X} \mid G\gamma(1)=G\delta(0)\}.
\]
The notation reflects the fact that this space is the topological pullback of the maps $q\circ\mathrm{ev}_0,q\circ\mathrm{ev}_1:\paths{X}\to X/G$ obtained by composing the two end-point evaluation maps $\mathrm{ev}_0,\mathrm{ev}_1:\paths{X}\to X$ with the projection $q:X\to X/G$ onto the orbit space. One can think of $\paths{X}\times_{X/G}\paths{X}$ as a space of broken paths in $X$, which are continuous except for at one point where they are allowed to ``jump'' to a point in the same $G$-orbit. Now the product $G\times G$ acts on $\paths{X}\times_{X/G}\paths{X}$ via $(g,h)(\gamma,\delta)=(g\gamma,h\delta)$, and the projection map 
\[
\evalG{X}:\paths{X}\times_{X/G}\paths{X}\to X\times X,\qquad (\gamma,\delta)\mapsto (\gamma(0),\delta(1))
\]
is $(G\times G)$-equivariant.

\begin{defn}[{\cite{LubMar}}]\label{def:TC^G}
Let $X$ be a $G$-space. The \emph{invariant topological complexity} of $X$, denoted $\TC^G(X)$, is defined to be the minimal integer $k$ such that $X\times X$ may be covered by $(G\times G)$-invariant open sets $U_0,\ldots , U_k$, each admitting a $(G\times G)$-map $s_i:U_i\to \paths{X}\times_{X/G}\paths{X}$ such that $\evalG{X}\circ s_i = \iota_{U_i}:U_i\hookrightarrow X\times X$. Briefly, $\TC^G(X)=\secat_{G\times G}(\evalG{X})$.
\end{defn}

\begin{prop}[{\cite[Proposition 3.25]{LubMar}}]
The invariant topological complexity is $G$-homotopy invariant. That is, if $X\simeq_G Y$, then $\TC^G(X)=\TC^G(Y)$.
\end{prop}

This definition of $\TC^G(X)$ in terms of path spaces is one of several equivalent definitions given in \cite{LubMar}. Another is given in terms of a specialization of the equivariant $\mathscr{A}$-category of Clapp and Puppe \cite{ClappPuppe86}.

\begin{defn}[{\cite[Definition 2.2]{LubMar}}]
Let $X$ be a $G$-space, and let $A\subseteq X$ be a $G$-invariant subset. The \emph{equivariant $A$-category of $X$}, denoted ${}_A\cat_G(X)$, is defined to be the minimal $k$ such that $X$ may be covered by $G$-invariant open sets $U_0,\ldots , U_k$ such that each inclusion $\iota_{U_i}:U_i\hookrightarrow X$ is $G$-homotopic to a map with values in $A$.
\end{defn}

\begin{ex}
Let $G=\{e\}$ be the trivial group, let $X$ be any space, and let $\Delta(X)\subseteq X\times X$ be the diagonal. Then
\[
\TC(X)={}_{\Delta(X)}\cat_{\{e\}}(X\times X)=:{}_{\Delta(X)} \cat (X\times X).
\]
\end{ex}

Note that for a $G$-space $X$ the diagonal $\Delta(X)\subseteq X\times X$ does not carry a natural action of $G\times G$. Lubawski and Marzantowicz introduce the \emph{saturated diagonal}
\[
\daleth(X):=(G\times G)\Delta(X) = \{(x,gx) \mid x\in X, g\in G\},
\]
which is a $(G\times G)$-invariant subset $\daleth(X)\subseteq X\times X$. 
 
 \begin{prop}[{\cite[Lemmas 3.5, 3.8]{LubMar}}]\label{prop:ClappPuppe}
 Let $X$ be a $G$-space. Then:
 \begin{enumerate}
 \item $\TC_G(X)={}_{\Delta(X)} \cat_G(X\times X)$;
 \item $\TC^G(X)={}_{\daleth(X)} \cat_{G\times G}(X\times X)$.
 \end{enumerate}
 \end{prop}
 
 \begin{rem}
 The arguments used to prove Proposition \ref{prop:ClappPuppe} show more generally that if $A\subseteq X$ is a $G$-invariant subset and $p:A'\to X$ is a $G$-fibrational substitute of the inclusion $\iota_A:A\subseteq X$, then ${}_A\cat_G(X)=\secat_G(p)$. 
 \end{rem}

Analogously to the non-equivariant case, there is a Whitehead definition of ${}_A\cat_G(X)$ in terms of a ``fat wedge'' construction. Given a $G$-invariant subset $A\subseteq X$, let
\[
F_A^k(X):=\{(x_0,\ldots , x_k) \in X^{k+1} \mid x_i\in A\mbox{ for some }0\leq i\leq k\},
\]
the \emph{$(k+1)$-fold fat $A$-sum}. Note that $F_A^k(X)\subseteq X^{k+1}$ is $G$-invariant under the diagonal action of $G$ on $X^{k+1}$.  

\begin{defn}
The \emph{Whitehead equivariant $A$-category of $X$}, denoted ${}_A\cat_G^{\rm Wh}(X)$, is defined to be the minimal integer $k$ such that the $(k+1)$-fold diagonal map $\Delta_{k+1}:X\to X^{k+1}$ is $G$-homotopic to a map with values in $F_A^k(X)$.
\end{defn}

\begin{thm}[{\cite[Theorem 2.7]{LubMar}}]\label{thm:Whitehead}
Suppose $G$ is compact Lie, $X$ is a compact $G$-ANR, and the inclusion $\iota_A:A\hookrightarrow X$ is a closed $G$-cofibration. Then  
\[
{}_A\cat_G(X)={}_A\cat_G^{\rm Wh}(X).
\]
\end{thm}

A proof of Theorem \ref{thm:Whitehead} can be found in the preprint version \cite{LubMarPre} of \cite{LubMar}. The non-equivariant case was observed in the FU Berlin thesis of A.\ Fass\'{o} Velenik \cite{Fasso}. Note that the diagonal inclusion $\Delta(X)\hookrightarrow X\times X$ is a closed $G$-cofibration whenever $X$ is a metric $G$-ANR or $G$-CW complex, as follows from \cite{Lewis}, so in either of these cases one has
\[
\TC_G(X)={}_{\Delta(X)} \cat_G^{\rm Wh}(X\times X).
\]
The question of when the saturated diagonal inclusion $\daleth(X)\hookrightarrow X\times X$ is a $(G\times G)$-cofibration appears to be more subtle, and is left open in \cite{LubMar}. It is shown in \cite[Theorem 3.15]{LubMar} that if $G$ is a finite group and $X$ is a compact $G$-ANR, then $\daleth(X)\hookrightarrow X\times X$ is a $(G\times G)$-cofibration, and so
\[
\TC^G(X)={}_{\daleth(X)} \cat_{G\times G}^{\rm Wh}(X\times X)
\]
under these hypotheses. The Whitehead definitions are purely homotopical in that they do not involve open sets, and so different proof techniques can be applied. The authors apply the Whitehead definitions to prove the following product formulae.

\begin{thm}[{\cite[Theorem 3.18]{LubMar}}]\label{thm:productGH}
Let $G$ and $H$ be compact Lie groups, let $X$ be a compact $G$-ANR and let $Y$ be a compact $H$-ANR. Then
\[
\TC_{G\times H}(X\times Y)\leq \TC_G(X)+\TC_H(Y).
\]
Furthermore, if $\daleth(X)\subseteq X\times X$ is a $(G\times G)$-cofibration and $\daleth(Y)\subseteq Y\times Y$ is an $(H\times H)$-cofibration (for example if $G$ and $H$ are finite), then
\[
\TC^{G\times H}(X\times Y)\leq \TC^G(X) + \TC^H(Y).
\]
\end{thm}

As pointed out in \cite[Remark 3.20]{LubMar}, if $G=H$ and $X$ and $Y$ are $G$-spaces, the inequality
\[
\TC^G(X\times Y) \leq \TC^G(X)+\TC^G(Y)
\]
 for the diagonal action on $X\times Y$, analogous to Theorem \ref{thm:productTC_G}, is false in general. See Example \ref{ex:TC^GS1xS1} below. 

We now return to what apparently motivated the definition of invariant topological complexity: its relationship to the ordinary topological complexity of the orbit space. It is fairly easy to check from the definitions that 
\[
\TC(X/G)\leq \TC^G(X)
\]
for any $G$-space $X$. Using the Covering Homotopy Theorem of Palais (\cite{Palais}, see also \cite[Theorem II.7.3]{Bredon}) one can prove the opposite inequality in case $G$ is compact Lie and the action has a single orbit type.

\begin{thm}[{\cite[Theorem 3.10]{LubMar}}]\label{thm:TC^Gfree}
Let $G$ be a compact Lie group, and let $X$ be a $G$-space with a single orbit type. Then $\TC(X/G)=\TC^G(X)$.
\end{thm}

\begin{ex}\label{ex:TC^GS1xS1}
Let $G=S^1$ act on $X=S^1$ by rotation. Since the action is free and transitive, $\TC^G(X)=\TC(X/G)=0$. Now consider the diagonal $G$-action on the torus $X\times X=S^1\times S^1$. This action is still free, and the orbit space $(X\times X)/G$ is diffeomorphic to $S^1$. Therefore $\TC^G(X\times X)=\TC((X\times X)/G)=\TC(S^1)=1$. Note in particular that $\TC^G(X\times X)>\TC^G(X)+\TC^G(X)$. 
\end{ex}

\begin{ex} 
Let $G=C_2$ act on a closed orientable surface $\Sigma$ of positive genus, via an orientation-reversing free involution. After the calculation of the topological complexity of non-orientable surfaces due to Dranishnikov \cite{Dranish2016,Dranish2017} and Cohen--Vandembroucq \cite{CohenVan}, we have $\TC^G(\Sigma)=\TC(\Sigma/G)=4$. 
\end{ex}

\begin{ex}
Let $G=C_2$ act on the sphere $S^n$ via the antipodal map. Then $\TC^G(S^n)=\TC(\R P^n)$, and the latter equals the immersion dimension of the real projective space $\R P^n$ for $n\neq 1,3,7$, by the result of Farber--Tabachnikov--Yuzvinsky \cite{FTY}.
\end{ex}

In fact more is true: the invariant topological complexity is invariant under \emph{quotients} and \emph{induction}. Let us now explain what we mean by this. Let $X$ be a $G$-space. If $K\lhd G$ is a normal subgroup, one can check that the quotient group $G/K$ acts on the orbit space $X/K$ via $(gK)(Kx)= Kgx$. If $G\leq H$, one can induce an $H$-space from $X$ by taking $H\times_G X$, the orbit space of the diagonal action of $G$ on $H\times X$, where $G$ acts on $H$ via the group operation.

\begin{thm}[{\cite[Corollary 5.10]{ACGO}}]\label{thm:TC^GMorita}
Let $K\lhd G\le H$ be compact Lie groups and let $X$ be a metrizable $G$-space. Then
\[
\TC^{H}(H\times_G X)=\TC^G(X),
\]
and assuming $K$ acts freely on $X$,
\[
\TC^{G/K}(X/K)=\TC^G(X).
\]
\end{thm}

The above two properties imply that the invariant topological complexity is \emph{Morita invariant}. Two group actions $G\times X\to X$ and $H\times Y\to Y$ are {\em Morita equivalent} if the translation groupoids $G\ltimes X$ and $H\ltimes Y$ they define
are equivalent as orbifolds. It has recently been shown \cite{Pardon} that all orbifolds arise as orbit spaces of compact group actions, and therefore the invariant topological complexity can be used to define a topological complexity of orbifolds. We refer the reader to \cite{ACGO} and the subsequent paper \cite{AngelColman} by \'{A}ngel and Colman for more background and references on orbifolds and their Lusternik--Schnirelmann-type invariants. 

\begin{rem}
Theorem \ref{thm:TC^GMorita} is deduced as a special case of a more general result \cite[Theorem 5.1]{ACGO}, that asserts that the equivariant $\mathscr{A}$-category ${}_\mathscr{A}\cat_G(X)$ is Morita invariant, for certain families $\mathscr{A}$ of $G$-spaces. This applies also to show that the equivariant category $\cat_G(X)$ is Morita invariant \cite[Corollary 5.9]{ACGO}, and so also gives an orbifold invariant. The resulting invariant is compared in \cite{AngelColman} with a notion of \emph{orbifold category} defined in terms of open covers and orbifold homotopy, and the two notions are shown to coincide for $G$ finite.
\end{rem}

We next describe the relationship between invariant topological complexity and equivariant category.

\begin{prop}[{\cite[Proposition 3.23, Remark 3.24]{LubMar}}]
Let $X$ be any $G$-space and $x_0\in X$ any point. Then $\TC^G(X)\leq {}_{Gx_0\times Gx_0}\cat_{G\times G}(X\times X)$. In particular, if $x_0\in X^G$ is a fixed point and $X$ is $G$-connected, then $\TC^G(X)\leq 2\cat_{G}(X)$.
\end{prop}


\begin{prop}[{\cite[Proposition 2.7]{BlaKal2}}]
Let $X$ be a $G$-space with fixed point $x_0\in X^G$. Then $\cat_G(X)\leq \TC^G(X)$.
\end{prop}

\begin{cor} Let $X$ be a $G$-connected $G$-space with fixed point $x_0\in X^G$. Then $\cat_G(X)\leq \TC^G(X)\leq 2\cat_G(X)$. In particular, under these assumptions $\TC^G(X)=0$ if and only if $X$ is $G$-contractible.
\end{cor}

Recall that a $G$-space $X$ is called \emph{$G$-contractible} if the identity map $\mathrm{id}_X:X\to X$ is $G$-homotopic to a map with values in a single orbit, or equivalently, $\cat_G(X)=0$. As remarked in \cite[Corollary 2.8, Remark 2.9]{BlaKal2}, the fact that $G$-contractibility implies $\TC^G(X)=0$ is true without the assumptions of $G$-connectivity or the existence of a fixed point. The same is not true of $\TC_G(X)$, as Example \ref{ex:S1rotation} shows.

Finally, we give the relationship of the invariant topological complexity with the topological complexity of the fixed point sets of the action.

\begin{thm}[{\cite[Corollary 3.26]{LubMar}}]
For any $G$-space $X$, we have $\TC(X^G)\leq \TC^G(X)$.
\end{thm}

Note that, in contrast to Corollary \ref{cor:fixed} for the equivariant topological complexity, it is not true that $\TC(X^H)\leq \TC^G(X)$ for all subgroups $H\leq G$. As a consequence, $\TC^G(X)$ may be finite, even if $X$ fails to be $G$-connected. The following example illustrates these facts.

\begin{ex}[{\cite[Example 2.10]{BlaKal2}}]
Let $G=C_2\times C_2$ be the product of two cyclic groups of order two; call the generators of these cyclic factors $a$ and $b$. Define an action of $G$ on $X=S^1\setminus\{i,-i\}\subseteq \C$ by setting $a(x+iy)=-x+iy$ and $b(x+iy)=x-iy$. In other words, $G$ acts on the circle with North and South pole removed, via reflections in the coordinate axes. One can easily imagine that the identity map $\mathrm{id}_X:X\to X$ is $G$-homotopic to a map with values in the $G$-orbit $\{1,-1\}$, hence $X$ is $G$-contractible. It follows that $\TC^G(X)=0$. On the other hand, $\TC(X^{\langle b\rangle})=\TC(\{-1,1\})=\infty$.
\end{ex}

\begin{rem}  B\l{}aszczyk and Kaluba in \cite{BlaKal2} investigate the equivariant and invariant topological complexities of $C_p$-spheres, where $C_p$ is the cyclic group of order $p$, a prime. They show that if the action is linear or semi-linear, then 
\[
1\leq \TC_{C_p}(S^n), \TC^{C_p}(S^n)\leq 2,
\]
while for arbitrary smooth actions both invariants will generically be at least $3$, and may be as large as $n-2$.
\end{rem}

\begin{rem} In \cite{BayehSarkar20} the authors define the {\em sequential invariant topological complexity} $\TC^{r,G}(X)$ for a $G$-space $X$ and $r\ge2$, and establish some natural generalizations of results in this section. Sequential versions of Theorem \ref{thm:TC^GMorita} are also established in \cite{ACGO}.
\end{rem}

\section{Strongly equivariant topological complexity}\label{S:TC^*_G}

A third version of topological complexity with symmetries was introduced by Dranishnikov in \cite{Dranish2015}, with the aim of giving improved upper bounds for the ordinary topological complexity.

\begin{defn}\cite{Dranish2015}\label{def:TC_G^*}
Let $X$ be a $G$-space. The \emph{strongly equivariant topological complexity} of $X$, denoted $\TC_G^*(X)$, is defined to be the minimal integer $k$ such that $X\times X$ can be covered by $(G\times G)$-invariant open sets $U_0,\ldots , U_k$, each admitting a $G$-map $s_i:U_i\to \paths{X}$ such that $\eval{X}\circ s_i=\iota_{U_i}:U_i\hookrightarrow X\times X$.
\end{defn}

\begin{prop}
The strongly equivariant topological complexity is $G$-homotopy invariant. That is, if $X\simeq_G Y$, then $\TC^*_G(X)=\TC^*_G(Y)$.
\end{prop}

\begin{proof} This is an easy adaptation of the proof of $G$-homotopy invariance of $\TC_G(X)$, which itself is an easy generalization of the proof of homotopy invariance of $\TC(X)$.
\end{proof}

Note that the only difference in Definition \ref{def:TC_G^*} compared to Definition \ref{def:TC_G} of $\TC_G(X)$ is that the sets $U_i$ are required to be $(G\times G)$-invariant, rather than just invariant under the action of the diagonal subgroup $G\cong\Delta(G)$. Hence $\TC_G(X)\leq \TC_G^*(X)$ is obvious. Finding examples of strict inequality seems to be a subtle problem. One needs a lower bound for $\TC^*_G(X)$ which is not a lower bound for $\TC_G(X)$. 

To this end, we offer the following cohomological lower bound for $\TC^*_G(X)$ when $G$ is finite. In that case, $G=\Delta(G)\leq G\times G$ is a finite index subgroup, and there results a \emph{transfer map}
\[
\operatorname{tr}: H^*_G(Y)\to H^*_{G\times G}(Y)
\]
in Borel equivariant cohomology for any $(G\times G)$-space $Y$. If $E=E(G\times G)$ is a contractible free $(G\times G)$-space, then this is the transfer associated to the finite cover $E\times_G Y\to E\times_{G\times G}Y$. Here and elsewhere below, we take coefficients in an arbitrary commutative ring $R$ which is omitted from the notation. More generally, we could make a similar statement in any multiplicative equivariant cohomology theory with transfers. 

\begin{prop}\label{prop:lowerTC^*_G}
Let $X$ be a $G$-space with $G$ finite. Suppose there are cohomology classes $x_1,\ldots , x_k\in \ker\big(\Delta^*: H^*_G(X\times X)\to H^*_G(X)\big)$ such that  
\[
0\neq \operatorname{tr}(x_1)\cdots \operatorname{tr}(x_k)\in H^*_{G\times G}(X\times X).
\]
Then $\TC_G^*(X)\ge k$.
\end{prop}

\begin{proof}
Assume for a contradiction that $\TC_G^*(X)<k$. Then we have a cover of $X\times X$ by $(G\times G)$-invariant open sets $U_1,\ldots , U_k$, each of which admits a $G$-map $s_i:U_i\to \paths{X}$ such that $\eval{X}\circ s_i=\iota_{U_i}$. It follows that for each $i=1,\ldots,k$ the cohomology class $x_i$ in the statement is in the kernel of the restriction-induced map $H^*_G(X\times X)\to H^*_G(U_i)$. By naturality of the transfer we have a commuting square
\[
\begin{tikzcd}
H^*_G(X\times X) \ar[r]  \ar[d,"\operatorname{tr}"] & H^*_G(U_i) \ar[d,"\operatorname{tr}"] \\
H^*_{G\times G}(X\times X) \ar[r] & H^*_{G\times G}(U_i).
\end{tikzcd}
\]
The additivity of the transfer then implies that $\operatorname{tr}(x_i)$ is in the kernel of the restriction-induced map $H^*_{G\times G}(X\times X)\to H^*_{G\times G}(U_i)$. The usual argument involving the long exact sequences of the pairs $(X\times X,U_i)$ and naturality of relative cup products completes the proof.
\end{proof}

\begin{rem}\label{rem:lowerTC^*_G}
 Note that we do not require that the product $x_1\cdots x_k$ is nonzero in $H^*_G(X\times X;R)$, and so the above does not constitute a lower bound for $\TC_G(X)$. Since the transfer is not multiplicative in general, this lower bound might in theory exceed the lower bound for $\TC_G(X)$ described in Theorem \ref{thm:lowerTC_G}. 

If we take coefficients for cohomology in a field whose characteristic does not divide the order of $G$, then $H^*_G(Y)\cong H^*(Y)^G$, the sub-algebra of $H^*(Y)$ consisting of elements fixed under the action of $G$. The transfer in  Proposition \ref{prop:lowerTC^*_G} then is given by
\[
\operatorname{tr}: H^*(X\times X)^G\to H^*(X\times X)^{G\times G},\qquad \operatorname{tr}\left(\sum_j \alpha_j \otimes \beta_j\right) = \sum_{g\in G}\sum_j \alpha_j\otimes g^*(\beta_j).
\]
We note that there can exist invariant zero-divisors $x\in \ker\big(\Delta^*:H^*(X\times X)^G\to H^*(X)^G\big)$ such that $\operatorname{tr}(x)\in  H^*(X\times X)^{G\times G}$ is not a zero-divisor. 
\end{rem}

Recall that for a principal $G$-bundle $E\to B$ and a $G$-space $X$, Theorem \ref{thm:TCFib} gives an upper bound for the topological complexity $\TC(E\times_G X)$ of the total space of the associated bundle with fibre $X$, in terms of the topological complexity $\TC(B)$ of the base and the equivariant topological complexity $\TC_G(X)$ of the fibre. By replacing $\TC_G(X)$ by $\TC_G^*(X)$, Dranishnikov is able to improve this multiplicative upper bound to an additive one. 

\begin{thm}[{\cite[Theorem 3.1]{Dranish2015}}]\label{thm:strongupper}
Let $E\to B$ be a principal $G$-bundle, and let $X$ be a proper $G$-space. Suppose further that $E\times_G X$ and $B$ are locally compact metric ANRs. Then
\[
\TC(E\times_G X)\leq \TC(B)+\TC^*_G(X).
\]
\end{thm}

Strongly equivariant topological complexity admits the following dimensional upper bound.

\begin{prop}[{\cite[Proposition 3.2]{Dranish2015}}]\label{prop:strongdim}
Suppose that a discrete group $\pi$ acts freely and properly on a simply-connected locally compact ANR space $Y$. Then $\TC^*_\pi(Y)\leq \dim(Y)$, where $\dim$ denotes the covering dimension.
\end{prop}

Recall that the topological complexity of a discrete group $\pi$ is defined by $\TC(\pi):=\TC(B\pi)$, where $B\pi=E\pi/\pi$ is the base space of a universal principal $\pi$-bundle.

\begin{thm}[{\cite[Theorem 3.3]{Dranish2015}}]
Let $X$ be a CW complex with fundamental group $\pi:=\pi_1(X)$. Then $\TC(X)\leq \TC(\pi)+\dim(X)$.
\end{thm}

\begin{proof}
Let $E\pi\to B\pi$ be a universal principal $\pi$-bundle and let $\tilde{X}$ be the universal cover of $X$. The action of $\pi$ on $\tilde{X}$ is free and proper, and the associated bundle $E\pi\times_\pi \tilde{X}$ is homotopy equivalent to $X$. Therefore Theorem \ref{thm:strongupper} and Proposition \ref{prop:strongdim} give
\[
\TC(X)=\TC(E\pi\times_\pi\tilde{X})\leq \TC(\pi)+\TC^*_\pi(\tilde{X})\leq \TC(\pi)+\dim(X).
\]
\end{proof}

When the fundamental group $\pi$ has torsion, $\TC(\pi)$ is infinite, and the bound $\TC(X)\leq \TC(\pi)+\TC^*_\pi(\tilde{X})$ contains no information. In the paper \cite{FGLO}, Farber, the author, Lupton and Oprea prove a strengthening of this bound which can be of use even when $\pi$ has torsion. 

\begin{thm}[{\cite[Theorem 3, Proposition 3.8]{FGLO}}]
Let $X$ be a locally finite cell complex with fundamental group $\pi$ and universal cover $\tilde{X}$. Then
\[
\TC(X)\leq \TC^\mathcal{D}(X) + \TC^*_\pi(\tilde{X}).
\]
Here $\TC^\mathcal{D}(X)$ is the \emph{$\mathcal{D}$-topological complexity} of $X$, defined to be $\secat(Q)$ where $Q:\tilde{X}\times_\pi \tilde{X}\to X\times X$ is the cover of $X\times X$ corresponding to the diagonal subgroup.
\end{thm}

In fact, from the results in \cite{FGLO} we can deduce the following dimension-connectivity estimate for the strongly equivariant topological complexity of the universal cover, which may improve on the estimate given by Proposition \ref{prop:strongdim}.

\begin{prop}
Let $X$ be a locally finite simplicial complex with fundamental group $\pi$. Suppose the universal cover $\tilde{X}$ is $k$-connected, where $k\geq1$. Then
\[
\TC^*_\pi(\tilde{X})\leq\left\lceil\frac{2\dim(X) - k}{k+1}\right\rceil.
\]
\end{prop}

In the case of free actions of discrete groups, the strongly equivariant topological complexity is bounded above by the invariant topological complexity.

\begin{prop}[{\cite[Proposition 4.29]{AngelColmanSurvey}}]\label{prop:strongcovering}
Let $\pi$ be a discrete group acting freely and properly discontinuously on a simply-connected space $X$. Then 
\[
\TC^*_\pi(X)\leq \TC(X/\pi)\leq\TC^\pi(X).
\]
\end{prop}

The conditions on the action in Proposition \ref{prop:strongcovering} are to ensure that the maps $X\times X\to X\times_\pi X$ and $X\times_\pi X\to X/\pi\times X/\pi$ are covering maps, and the homotopy lifting property for covering maps is used heavily in the proof. Even for a proper free action of a compact Lie group $G$ on a simply-connected space $X$, the inequality $\TC_G^*(X)\leq \TC^G(X)$ need not hold. For example, let $G$ be a non-contractible simply-connected Lie group (such as $S^3$) acting on itself by translation. Then $\TC^G(G)=\TC(G/G)=0$ by Theorem \ref{thm:TC^Gfree}, while $\TC^*_G(G)\geq\TC_G(G)\geq\TC(G)>0$.  

\begin{rem} In \cite{PaulSen}, Paul and Sen have defined the \emph{sequential strongly equivariant topological complexity} $\TC_{r,G}^*(X)$ for a $G$-space $X$ and $r\ge2$, and proved generalizations of many of the results in this section.
\end{rem}

\section{Effective topological complexity}\label{S:TCeff}

The fourth and final version of topological complexity with symmetry we shall consider is the \emph{effective topological complexity} of B\l{}aszczyk and Kaluba \cite{BlaKal}. Although it shares some commonalities with the invariant topological complexity, the effective topological complexity is quite different from the other three in that it does not require motion planners to be equivariant. Instead the symmetries are used only to reduce the complexity of the motion planning task.

Given a $G$-space $X$ and natural number $n\geq 1$, let
\[
\mathcal{P}_n(X):=\{(\gamma_1,\ldots, \gamma_n)\in \paths{X}^n \mid G\gamma_i(1)=G\gamma_{i+1}(0)\mbox{ for }i=1,\ldots ,n-1\}.
\]
Note that $\mathcal{P}_1(X)=\paths{X}$ and $\mathcal{P}_2(X)=\paths{X}\times_{X/G}\paths{X}$ is the space of broken paths appearing in the definition of $\TC^G(X)$. We think of $\mathcal{P}_n(X)$ as paths in $X$ that are continuous except for at $n-1$ points, where they are allowed to ``jump'' to a point in the same $G$-orbit. The evaluation map $\evaln{X}:\mathcal{P}_n(X)\to X\times X$ defined by $\evaln{X}(\gamma_1,\ldots, \gamma_n)=\big(\gamma_1(0),\gamma_n(1)\big)$ is shown to be a fibration in \cite{BlaKal}.

\begin{defn}[{\cite[Definition 3.1]{BlaKal}}]
Given a $G$-space $X$ and natural number $n\ge 1$, define $\TC^{G,n}(X)$ to be the minimal integer $n$ such that $X\times X$ can be covered by open sets $U_0,\ldots , U_k$, each admitting a continuous map $s_i:U_i\to \mathcal{P}_n(X)$ such that $\evaln{X}\circ s_i=\iota_{U_i}:U_i\hookrightarrow X\times X$. In short, $\TC^{G,n}(X)=\secat(\evaln{X})$.
\end{defn}

It is clear that $\TC^{G,1}(X)=\TC(X)$. Since the maps $s_i$ are not required to be $(G\times G)$-equivariant, we also have $\TC^{G,2}(X)\leq\TC^G(X)$.

\begin{lem}[{\cite[Lemma 3.2]{BlaKal}}]
One has $\TC^{G,n+1}(X)\leq \TC^{G,n}(X)$ for all $n\geq 1$.
\end{lem}

Thus $\{\TC^{G,n}(X)\}_{n\ge1}$ is a decreasing sequence of integers, bounded below by $0$. It follows that there exists $n_0$ such that $\TC^{G,n}(X)=\TC^{G,n_0}(X)$ for all $n\ge n_0$.

\begin{defn}[{\cite[Definition 3.4]{BlaKal}}]
Let $n_0\in \N$ be such that $\TC^{G,n}(X)=\TC^{G,n_0}(X)$ for all $n\ge n_0$. The \emph{effective topological complexity} of the $G$-space $X$ is defined to be $\TC^{G,\infty}(X):=\TC^{G,n_0}(X)$.
\end{defn}

\begin{prop}[{\cite[Theorem 3.3]{BlaKal}}]
If $X\simeq_G Y$, then $\TC^{G,n}(X)=\TC^{G,n}(Y)$ for all $n\in \N$. In particular, the effective topological complexity is $G$-homotopy invariant.
\end{prop}

It is evident that 
\[
\TC^{G,\infty}(X)\leq \TC^{G,1}(X)=\TC(X)\leq \TC_G(X)\leq\TC_G^*(X).
\]
and that
\[
\TC^{G,\infty}(X)\leq\TC^{G,2}(X)\leq \TC^G(X).
\]
Hence effective topological complexity is less than or equal to all of the other equivariant topological complexities discussed in this chapter.

For free $G$-actions one can take $n_0=2$, giving $\TC^{G,\infty}(X)=\TC^{G,2}(X)$, as shown in \cite[Theorem 5.1]{BlaKal} (the assumption that $G$ be finite is unnecessary).

\begin{rem} In \cite{C-AGGI-Z} the authors revisit the definition of effective topological complexity, defining a variant which they denote $\TC_{\sf effv}^G(X)$, and which agrees with $\TC^{G,\infty}(X)$ for free actions. In essence, one replaces the broken path space $\mathcal{P}_n(X)$ by the space of tuples
\[
\{(\gamma_1,g_1,\ldots, g_{n-1},\gamma_n)\in (\paths{X}\times G)^{n-1}\times\paths{X} \mid g_i\gamma_i(1)=\gamma_{i+1}(0)\mbox{ for }i=1,\ldots, n-1\},
\]
thereby keeping track of the group element used to ``jump'' within the $G$-orbit at each break. This has the nice effect that the resulting sequence of sectional categories stabilizes at $n=2$ even for non-free actions, so that $\TC_{\sf effv}^G(X)$ can be defined simply to be the sectional category of the evaluation map $(\gamma_1,g_1,\gamma_2)\mapsto \big(\gamma_1(0),\gamma_2(1)\big)$. Following a suggestion of Pave\v{s}i\'{c} \cite{Pavesic}, they also define the \emph{effectual topological complexity}, denoted $\TC_{\sf effl}^G(X)$, to be the sectional category of the evaluation map $\paths{X}\to X\times (X/G)$ given by $\gamma\mapsto \big(\gamma(0), G\gamma(1)\big)$. It is shown in \cite[Theorem 1.1]{C-AGGI-Z} that if $G$ is a discrete group acting properly discontinuously on a Hausdorff space $X$ then
\[
\TC_{\sf effv}^G(X)\leq \TC_{\sf effl}^G(X) \leq \TC(X/G),
\]
and in fact both inequalities above can be strict, as happens for the antipodal involution on the $2$-torus \cite[Theorem 1.2]{C-AGGI-Z}.
\end{rem}

\begin{prop}[{\cite[Proposition 3.6]{BlaKal}}]
If  the $G$-space $X$ is $G$-contractible, then $\TC^{G,\infty}(X)=0$.
\end{prop}

The converse is false, as shown by the following example, which also illustrates that the effective topological complexity is not bounded below by the equivariant category. The problem of characterizing exactly which $G$-spaces have $\TC^{G,\infty}(X)=0$ is stated in \cite[Section 7.2]{BlaKal}, and appears to still be open.

\begin{ex}
Let $G$ be a discrete group, and let $EG$ be a contractible free $G$-space. Since $\TC^{G,\infty}(EG)\leq \TC(EG)=0$,
we have $\TC^{G,\infty}(EG)=0$. However the $G$-space $EG$ is not $G$-contractible. In fact
\[
\cat_G(EG)=\cat(EG/G)=\cat(BG)=\cld(G),
\]
where $\cld(G)$ denotes the cohomological dimension of $G$. This is infinite if $G$ is finite, and positive whenever $G$ is non-trivial.

This example also shows that $\TC^{G,\infty}(X)$ does not in general coincide with $\TC(X/G)$ for free $G$-actions.
\end{ex}

Since $\TC^{G,n}(X)=\secat(\evaln{X}:\mathcal{P}_n(X)\to X\times X)$, it admits a cohomological lower bound in the form of the nilpotency of the kernel of the map induced by $\evaln{X}$ in cohomology. Using transfer arguments, B\l{}aszczyk and Kaluba are able to prove the following results.

\begin{thm}[{\cite[Theorem 4.1]{BlaKal}}]\label{thm:zdcleff}
Let $G$ be a finite group and let $X$ be a $G$-CW complex. If $\Bbbk$ is a field of characteristic zero or $p$ prime to $|G|$, then 
\[
\TC^{G,\infty}(X)\geq \nil\ker\big(\cup:H^*(X/G;\Bbbk)\otimes H^*(X/G;\Bbbk)\to H^*(X/G;\Bbbk)\big).
\]
\end{thm}

\begin{cor}[{\cite[Corollary 4.2]{BlaKal}}]\label{cor:eff=TC}
Let $G$ be a finite group and let $\Bbbk$ be a field of characteristic zero or $p$ prime to $|G|$. If $G$ acts trivially on $H^*(X;\Bbbk)$ and $$\TC(X)=\nil\ker\big(\cup: H^*(X;\Bbbk)\otimes H^*(X;\Bbbk)\to H^*(X;\Bbbk)\big),$$ then $\TC^{G,\infty}(X)=\TC(X)$.
\end{cor}

B\l{}aszczyk and Kaluba compute the effective topological complexity of all linear $C_p$-actions on spheres, and some non-linear ones. The result quoted below follows easily from Corollary \ref{cor:eff=TC} on taking rational cohomology. For the remaining cases we refer to \cite[Corollary 5.10]{BlaKal}.

\begin{prop}[{\cite[Proposition 5.3]{BlaKal}}]
Suppose the cyclic group $C_p$ acts on the sphere $S^n$, where $p$ is prime and $n\ge1$. If $p>2$, or if $p=2$ and the action is orientation preserving, then
\[
\TC^{C_p,\infty}(S^n) = \TC(S^n) = \begin{cases} 1 & \mbox{ if $n$ odd,} \\ 2 & \mbox{ if $n$ even.} \end{cases}
\]
\end{prop}

As mentioned above, for free actions $$\TC^{G,\infty}(X)=\TC^{G,2}(X)=\secat(\pi_{2,X}:\mathcal{P}_2(X)\to X\times X).$$ For free actions with $G$ finite, one easily checks that there is a commutative diagram
\[
\begin{tikzcd}
\mathcal{P}_2(X) \arrow[rd,"\pi_{2,X}"] & \\
\daleth(X) \arrow[u,"\simeq"] \arrow[hook,r,"j"] & X\times X
\end{tikzcd}
\]
in which the vertical arrow is a homotopy equivalence. This motivates the following definition.

\begin{defn}[{\cite[Definition 7.3]{C-AG}}]
Let  $j:\daleth(X)\hookrightarrow X\times X$ be the inclusion of the saturated diagonal. An element in the kernel of $j^*:H^*(X\times X)\to H^*(\daleth(X))$ is called an \emph{effective zero-divisor}, where we take cohomology with arbitrary, possibly twisted, coefficients.
\end{defn}

\begin{lem}
Let $R$ be a commutative ring, and let $X$ be a free $G$-space with $G$ finite. Then $\TC^{G,\infty}(X)$ is bounded below by the \emph{effective zero-divisors cup-length} with $R$ coefficients. In symbols,
\[
\TC^{G,\infty}(X)\geq\nil\ker\big( j^*: H^*(X\times X;R)\to H^*(\daleth(X);R)\big).
\]
\end{lem}
 
 The above lemma is implicit in \cite{C-AG}, where the authors are motivated by potentially giving a new calculation of the topological complexity of non-orientable surfaces using effective topological complexity. Recall that the non-orientable surface of genus $g+1$ may be written as the orbit space $N_{g+1}=\Sigma_g/C_2$ of a free, orientation-reversing involution on $\Sigma_g$, the orientable surface of genus $g$. It follows that 
 \[
 \TC^{C_2,\infty}(\Sigma_g)\leq \TC^{C_2}(\Sigma_g)=\TC(\Sigma_g/C_2)=\TC(N_{g+1}),
 \]
 which raises the possibility of using the effective zero-divisors cup-length of $\Sigma_g$ to bound $\TC(N_{g+1})$ from below. In the low genus cases $g=0,1$ one has
 \[
\TC^{C_2,\infty}(\Sigma_g)=g+1<g+3=\TC(N_{g+1}),
\]
as follows from results of B\l{}aszczyk--Kaluba \cite{BlaKal} and Farber--Tabachnikov--Yuzvinsky \cite{FTY} in the $g=0$ case and Gonz\'alez--B\l{}aszczyk--Kaluba (unpublished) and Cohen--Vandembroucq \cite{CohenVan} in the $g=1$ case. For higher genera, the authors of \cite{C-AG} show the following using effective zero-divisors cup-length with $\Z$ coefficients. 

\begin{thm}[{\cite[Theorem 1.1]{C-AG}}]\label{thm:TCeffsurfaces}
Let $g\ge2$. For the antipodal involution on $\Sigma_g$, one has $\TC^{C_2,\infty}(\Sigma_g)\ge3$. Therefore, {\it a posteriori},
\[
3\leq \TC^{C_2,\infty}(\Sigma_g)\leq \TC(N_{g+1})=4.
\]
\end{thm}

Although the above is proved by exhibiting a nonzero product in untwisted integral cohomology, showing that the three classes involved are effective zero-divisors requires some non-trivial calculations in the group cohomology of orientable surface groups (with twisted integral coefficients). The authors express optimism in \cite[Remark 7.7]{C-AG} that using effective zero-divisors with twisted coefficients might lead to a proof that $\TC^{C_2,\infty}(\Sigma_g)\ge4$ in sufficiently high genera, and hence to an alternative proof that $\TC(N_{g+1})=4$.

Finally in this section we discuss the product inequality for effective topological complexity, which is analogous to Theorem \ref{thm:productGH}.

\begin{thm}[{\cite[Theorem 6.1]{BlaKal}}]\label{thm:producteff}
Let $X$ be a paracompact $G$-space and $Y$ be a paracompact $H$-space, where $G$ and $H$ are topological groups. Give $X\times Y$ the product action of $G\times H$. Then
\[
\TC^{G\times H,\infty}(X\times Y)\leq \TC^{G,\infty}(X) + \TC^{H,\infty}(Y).
\]
\end{thm}

If $G=H$ and we replace the product action in the above with the diagonal $G$-action, then the analogous inequality is false in general, as illustrated by the following example.

\begin{ex}[{\cite{BlaKal}}]
Let $G=C_2$ act on $S^1\subseteq \C$ via complex conjugation. Then $\TC^{C_2,\infty}(S^1)=0$: an effective motion planner on $X\times X$ is described in \cite[Proposition 5.7]{BlaKal}. The diagonal $C_2$-action on $S^1\times S^1$ is best visualized as a rotation through $\pi$ radians of the standard embedded torus in $\R^3$ around an axis (think of barbecuing a donut). The orbit space $(S^1\times S^1)/C_2$ is an orbifold homeomorphic to $S^2$ with $4$ singular points (think of a pillowcase). Now an application of Theorem \ref{thm:zdcleff} with rational coefficients gives
\[
\TC^{C_2,\infty}(S^1\times S^1)\geq 2 >0 = \TC^{C_2,\infty}(S^1) + \TC^{C_2,\infty}(S^1).
\]
\end{ex}

\begin{rem} The preprint \cite{BalTorres} by Balzer and Torres-Giese introduces sequential versions of the effective and effectual topological complexities, as well as a further variant they call the \emph{orbital topological complexity}. Basic properties are derived, and computations are given for involutions on spheres and surfaces.
\end{rem}

\section{Problems}\label{S:Problems}

We conclude this survey with a selection of open problems, which we hope will stimulate further research.

\begin{prob}
Give an example of a $G$-space $X$ and multiplicative $G$-equivariant cohomology theory $h^*_G$ for which
\[
\nil\ker \big(\Delta^*:h^*_G(X\times X)\to h^*_G(X)\big)>\max\{\TC(X^H) \mid H\leq G\},
\]
that is, for which the lower bound for $\TC_G(X)$ coming from equivariant cohomology exceeds the lower bound coming from the topological complexity of the fixed sets. 
\end{prob}

\begin{prob}
Give an example of a $G$-space $X$ for which $\TC^*_G(X)>\TC_G(X)$. (Proposition \ref{prop:lowerTC^*_G} and Remark \ref{rem:lowerTC^*_G} may be useful in this regard.)
\end{prob}

\begin{prob}[Z.\ B\l{}aszczyk, M.\ Kaluba]
For which $G$-spaces $X$ is $\TC^{G,\infty}(X)=0$?
\end{prob}

\begin{prob}[Z.\ B\l{}aszczyk, J.\ Gonz\'alez, M.\ Kaluba]
Compute $\TC^{C_2,\infty}(\Sigma_g)$ for the antipodal involution on an orientable surface of genus $g\ge2$. 
\end{prob}

\begin{prob}[Z.\ B\l{}aszczyk, M.\ Kaluba \cite{BlaKal2}]
Is it true that if $S^n$ is a smooth $C_p$-sphere with non-empty and path connected
fixed point set, then $S^n$ is equivariantly equivalent to a linear $C_p$-sphere
if and only if $\TC^{C_p}(S^n)\le 2$? (By Smith Theory, it would suffice to show that $\TC(\Sigma)\ge 3$ for any mod $p$ homology sphere $\Sigma$.)
\end{prob}

Let $\pi$ be a discrete group equipped with an action of a finite group $G$ by automorphisms. Briefly, we call $\pi$ a \emph{$G$-group}.

\begin{prob}
Define the equivariant topological complexity $\TC_G(\pi)$ of the $G$-group $\pi$. Prove equivariant analogues of the results in \cite{GLO} and \cite{FGLOaspherical}.
\end{prob}

\begin{prob}
Define $\TC^G(\pi)$, $\TC^*_G(\pi)$ and $\TC^{G,\infty}(\pi)$. Do any of these admit an algebraic description in terms of equivariant group cohomology? (Compare \cite{GrantMeirPatchkoria}, where $\cat_G(\pi)$ is defined and shown to be equal to an equivariant cohomological dimension $\cat_G(\pi)$.)
\end{prob}

\end{document}